\documentclass[12pt]{amsart}

\usepackage{amssymb}
\usepackage{graphicx}

\usepackage{color}

\newtheorem{theorem}{Theorem}

\newtheorem{proposition}[theorem]{Proposition}
\newtheorem{corollary}[theorem]{Corollary}

\newtheorem{them}{Theorem}
\newtheorem{lema}[them]{Lemma}

\theoremstyle{definition}

\newtheorem{example}[theorem]{Example}

\theoremstyle{remark}
\newtheorem{remark}{Remark}

\begin{document}

%\title[Hypo-efficient domination and hypo-unique domination graphs]{Hypo-efficient domination and hypo-unique domination graphs}
\title[]{Hypo-efficient domination and hypo-unique domination }

\author[]{Vladimir Samodivkin}
\address{Department of Mathematics, UACEG, Sofia, Bulgaria}
\email{vl.samodivkin@gmail.com}
\today
\keywords{domination number; efficient domination; unique domination; hypo-property}

\begin{abstract}
For a graph $G$ let $\gamma (G)$ be its domination number. 
We define a graph G to be 
(i)  a  hypo-efficient domination graph (or  a  hypo-$\mathcal{ED}$ graph)  
      if $G$ has no  efficient dominating set (EDS)  but every graph formed
			by removing a single vertex from $G$ has at least one EDS, and 
(ii)  a hypo-unique domination graph (a hypo-$\mathcal{UD}$ graph) if 
      $G$ has at least two  minimum dominating sets,
			but $G-v$ has a unique minimum dominating set  for each $v\in V(G)$. 
We  show that each  hypo-$\mathcal{UD}$ graph $G$ of order at least $3$
 is   connected  and $\gamma(G-v) < \gamma(G)$ for all $v \in V$. 
We obtain a tight  upper bound  on the order of a hypo-$\mathcal{P}$ graph 
in terms of the domination number and maximum degree of the graph, 
where $\mathcal{P} \in \{\mathcal{UD}, \mathcal{ED}\}$. 
 Families of  circulant graphs which achieve these bounds are presented. 
We also prove that the bondage number of any  hypo-$\mathcal{UD}$ graph 
is not more than the minimum degree plus one. 
 \end{abstract}

\maketitle

{\bf  MSC 2010}: 05C69

%\linenumbers

\section{Introduction}
All graphs considered in this article are finite, undirected, without loops or multiple edges. 
 For the graph theory terminology not presented here, we follow Haynes et al.   \cite{hhs1}. 
We denote the vertex set and the edge set of a graph $G$ by $V(G)$ and $ E(G),$  respectively. 
The {\em complement} $\overline{G}$ of $G$ is the  graph
whose vertex set is $V(G)$ and whose edges are the pairs of
nonadjacent vertices of $G$. 
The {\em join}  of  graphs $G$ and $H$, written $G \vee H$, is the 
graph obtained from the disjoint union of $G$ and $H$ by adding the edges $\{xy \mid x\in V(G), \ y \in V(H)\}$. 
In a graph $G$, for a subset $S \subseteq V(G)$ the {\em subgraph induced} by $S$ is the graph $\left\langle S\right\rangle$ 
with vertex set $S$ and edge set $\{xy \in E(G)\colon x,y \in S\}$. 
 We write $K_n$ for the {\em complete graph} of order $n$ and $C_n$ for a {\em cycle} of length $n$. 
Let $P_m$	 denote the {\em path} with $m$ vertices. 
	For any vertex $x$ of a graph $G$,  $N_G(x)$ denotes the set of all  neighbors of $x$ in $G$,  
	$N_G[x] = N_G(x) \cup \{x\}$ and the {\em degree} of $x$ is $deg_G(x) = |N_G(x)|$. 
		The {\em minimum} and {\em maximum} degree of a graph $G$ are denoted by $\delta(G)$ and $\Delta(G)$, respectively. 
	  A {\em leaf} of a graph is a vertex of degree $1$, while a  {\em support vertex} is a vertex adjacent to a leaf. 
For a subset $A \subseteq V (G)$, let $N_G[A] = \cup_{x \in A} N_G[x]$. 
The {\em coalescence} of disjoint graphs $H$ and $G$  is the graph $H \cdot G$
 obtained by identifying one vertex of $H$ and one vertex of $G$. 
%The identified vertex thus becomes a cut vertex of $H \cdot G$.

A set $D$ of vertices in a graph $G$  {\it dominates} a vertex $u \in V(G)$ 
if either $u \in D$ or $u$ is adjacent to some $v \in D$.
If $D$ dominates all vertices in a subset $T$ of $V(G)$ we say that $D$ {\it dominates} $T$. 
When $D$ dominates $V(G)$, $D$ is called a {\it dominating set} of the graph $G$. 
 That is, $D$ is a dominating set if and only if $N[D]=V(G)$. 
The {\em domination number} $\gamma (G)$ equals the minimum cardinality of a dominating
set in $G$, and a dominating set of $G$ with cardinality $\gamma (G)$ is called a $\gamma$-{\em set} of $G$. 
 A dominating set $D$ is called an {\it efficient dominating set} (EDS)  
 if $D$ dominates every vertex exactly once (\cite{bbs0}). 
A vertex $v$ of a graph $G$ is $\gamma$-{\em critical} if  $\gamma(G- v) <  \gamma(G)$. 
We denote by $V^-(G)$ the set of all  $\gamma$-critical vertices of $G$. 
A graph $G$ is a {\em vertex domination-critical graph} (or a {\em vc-graph}) if $V^-(G) = V(G)$ (\cite{bcd}). 
   The concept of domination in graphs has many applications to several fields. 
 Domination naturally arises in facility location problems,  
 in problems involving finding sets of representatives, in monitoring communication
or electrical networks, and in land surveying.  Many variants of the basic concepts
of domination have appeared in the literature. We refer to \cite{hhs1,hhs2}  for a survey of the area. 

 Let $\mathcal{I}$ denote the set of all mutually nonisomorphic graphs.
 A  graph property is any non-empty subset of $\mathcal{I} $. 
We say that a  graph $G$ has the  property $\mathcal{P}$ whenever 
there exists a graph $H \in \mathcal{P} $ which is isomorphic to $G$.
Any  set $S \subseteq V(G)$ such that the induced subgraph $G[S]$  possesses
	the property $\mathcal{P}$ is called a $\mathcal{P}$-set.  
%	For a survey on this subject we refer to Borowiecki et al. \cite{bbfms}. 
	
If a graph $G$ does not possess a given property $\mathcal{P}$, 
and for each vertex $v$ of $G$ the graph $G-v$ has property $\mathcal{P}$,
 then $G$ is said to be a {\it hypo-$\mathcal{P}$ graph}.
This concept is  closely related to the well-known concept of "critical". 
 A graph $G$ is {\it critical with respect to the property} $\mathcal{H}$ if 
$G$ possesses property $\mathcal{H}$ but  for every $v \in V(G)$ 
the graph $G-v$ does not have property $\mathcal{H}$. 
Thus, if $\overline{\mathcal{P}}$ denote the negation of $\mathcal{P}$, 
then a graph $G$ is a hypo-$\mathcal{P}$ graph  if and only if $G$ is critical with respect to 
property $\overline{\mathcal{P}}$. 
Nevertheless,  with some properties, however, it is more natural to consider 
the "hypo" point of view rather than the "critical" approach. 
A number of studies have been made where $\mathcal{P}$ stands for the graph 
being hamiltonian (see \cite{z}  and references therein),
traceable (see \cite{aw}  and references therein), planar(\cite{wa}), outerplanar (\cite{m}),  
eulerian and randomly-eulerian (\cite{k}). 
Let us mention also hypomatchable (or factor-critical)  graphs  
(for a survey up to  2003 see \cite{p}).
 Here  we focus on the case when  $\mathcal{P} \in \{\mathcal{ED}, \mathcal{UD} \}$, where
\begin{itemize}
\item[$\bullet$]  $\mathcal{ED} = \{H \in \mathcal{I}$ : $H$ {\em has an efficient dominating set}$\}$, and 
\item[$\bullet$]   $\mathcal{UD} = \{H \in \mathcal{I}$ : $H$ {\em has exactly one $\gamma$-set}$\}$.
\end{itemize}

More formally, we define:
 \begin{itemize}
\item[$\bullet$]  A graph $G$ is an {\em efficient domination graph} (or  an $\mathcal{ED}${\em-graph})  if $G$ has an EDS (\cite{kpy}).
\item[$\bullet$]  A graph $G$ is a {\em unique domination graph} (or a $\mathcal{UD}${\em-graph})  
                            if $G$ has exactly one $\gamma$-set. 
 \end{itemize}                           
 For results on graphs with a unique minimum dominating set see \cite{fi} and references therein. 
 
 \begin{itemize}                            
\item[$\bullet$]  A graph $G$ is a  {\em hypo-efficient domination graph} (or  an {\em hypo}-$\mathcal{ED}$ {\em graph})  
                            if $G$ has no  EDS but every graph formed by removing a single vertex from $G$ has at least one EDS. 
\item[$\bullet$]  A graph $G$ is a {\em hypo-unique domination graph}  (or a {\em hypo}-$\mathcal{UD}$ {\em graph}) if 
                                 $G$ has at least two $\gamma$-sets, but $G-v$ has a unique minimum dominating set  for each $v\in V(G)$.  
\end{itemize}

The paper is organized as follows. 
Section 2, contains some known results which are necessary to present our results.  
In  Section 3  we prove that each  hypo-$\mathcal{UD}$ graph of order at least $3$ is  a connected vc-graph 
and we obtain sharp upper bounds in terms of (a) domination number, and 
(b) domination number and  maximum degree for the order of a  hypo-$\mathcal{P}$ graph, 
where $\mathcal{P} \in \{\mathcal{UD}, \mathcal{ED}\}$. 
Families of  circulant graphs which achieve these bounds are presented. 
 We also prove that  the bondage number of any  hypo-$\mathcal{UD}$ graph 
is not more than the minimum degree plus one. 
We conclude in Section 4 with some open problems.

\section{Known results}

\begin{them} \label{effs1} \cite{bbs}
 Let $G$ be a graph.
If $G$ has vertex set $V(G) = \{v_1,v_2,..,v_n\}$, then $G$ has an EDS if and only if
 some subcollection  of $\{N[v_1], N[v_2],..,N[v_n]\}$ partitions $V(G)$.
 If $G$ has an EDS then the cardinality of any EDS of $G$ equals the domination number of $G$.
\end{them}

\begin{lema} \label{minus} \cite{bhns}
Let $G$ be a graph and $x,y \in V(G)$. 
If $x$ is  $\gamma$-critical then $\gamma(G-x) = \gamma(G)-1$ and  no vertex in $N_G[x]$ is in a $\gamma$-set of $G-x$.  
If $\gamma(G-y) >  \gamma(G)$ then $y$ is in all $\gamma$-sets of $G$.
\end{lema}

Theorem \ref{effs1}  and Lemma \ref{minus}  will be used in the sequel without specific   reference. 

\begin{them} \label{vc1} 
Let $G$ be a vc-graph. Then
\begin{itemize}
\item[(i)]  \cite{bcd}   $G$ is vc-critical if and only if each block of $G$ is vc-critical.
\item[(ii)]  \cite{bcd} $|V(G)| \leq (\Delta (G) +1)(\gamma(G) -1) +1$. 
                                        If the equality holds then \\ \cite{fhmg} $G$ is regular. 
\end{itemize}
 \end{them} 

\begin{remark}\label{2edge}
By Theorem \ref{vc1}(i), if $G$ is a connected nontrivial vc-graph then $G$ is $2$-edge connected and $\delta(G) \geq 2$.   
\end{remark}

The {\em corona} of  graphs $H$ and $K_1$ is the graph $H\circ K_1$ constructed from a copy
 of $H$, where for each vertex $v \in V(H)$, a new vertex $v^\prime$
 and a pendant edge $vv^\prime$ are added. Hence, $H\circ K_1$ has even order.  

\begin{them}\label{0.5}
Let a graph $G$ have no isolated vertices. Then 
(a) \cite{ore} $\gamma(G) \leq |V(G)|/2$ and  
(b) \cite{fjkr, px}  $\gamma(G) = |V(G)|/2$ 
 if and only if the components of $G$ are the cycle $C_4$
 or the corona $H\circ K_1$ for any connected graph $H$. 
\end{them}

 Let $\mathcal{A} = \{H_1,..,H_7\}$ be the collection of graphs in Figure \ref{fig:seven}.  
\begin{figure}[htbp]
	\centering
		\includegraphics[width=0.60\textwidth]{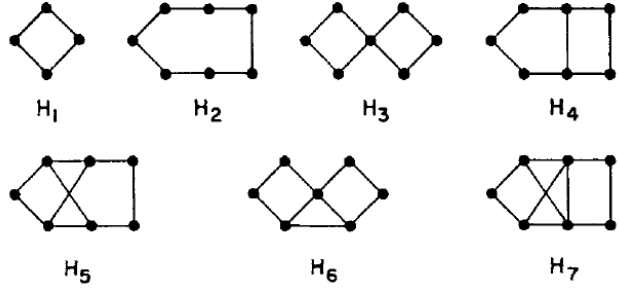}
	\caption{$\delta (H_i) \geq 2$ and $\gamma(H_i) > 2|V(H_i)|/5$, for $i=1,..,7$ (\cite{ms}). }
	\label{fig:seven}
\end{figure}

\begin{them}\cite{ms}\label{2/5}
If $G$ is a connected graph with $\delta(G) \geq 2$ and $G \not\in \mathcal{A}$, 
then $\gamma(G)\leq \frac{2}{5}|V(G)|$. 
\end{them}

Let $S = \{n_1, n_2, . . . , n_k\}$ be a set of integers such that 
$0 < n_1 <...<n_k <(n+1)/2$  and let the vertices of an $n$-vertex graph be
labelled $0,1,2,. . . , n - 1$.
Then the {\em circulant graph} $C(n, S)$ has $i \pm n_1, i \pm n_2 ,..., i \pm n_k$ $(\mbox{mod\ } n)$ 
 adjacent to each vertex $i$. 
If $n_k \not= n/2$ then $C(n, S)$ is regular of degree $2k$. 
 When $n_k = n/2$,  $C(n, S)$  is regular of degree $2k - 1$.

\begin{them}\label{circu}
Let $G  = C(n; \{1,2,..,k\})$, where $n \geq 3$ and $1\leq k < \left\lfloor n/2 \right\rfloor$. 
Then (a) \cite{grobler1} $\gamma (G) =  \left\lceil \frac{n}{2k+1} \right\rceil$, and 
(b) \cite{coe} $G$ is vc-critical if and only if $2k+1$ divides $n-1$. 
\end{them}

One measure of stability of the domination number is the bondage number. 
The bondage number $b(G)$ of a nonempty graph $G$ is the minimum number of
 edges whose removal of which from $G$ results in a graph with larger domination number. 
Recently, in 2013, Xu \cite{xu} gave a review article on bondage numbers. 

\begin{lema}\label{b=1}    \cite{t1}
If $G$ is a nontrivial graph with a unique minimum dominating set, then $b(G) = 1$.
\end{lema}

\newpage

\section{Hypo-unique and hypo-efficient domination}

We begin with results on hypo-$\mathcal{UD}$ graphs. 
Our first theorem shows that each  hypo-$\mathcal{UD}$ graph of order at least $3$ is  a connected vc-graph.

\begin{theorem}\label{udvc}
If $G$ is a hypo-$\mathcal{UD}$ graph then either $G=K_2$ or 
$G$ is a connected vc-graph with $|V(G)| \geq 4$. 
\end{theorem}
\begin{proof}
Let us assume that $G$ is not connected. Then $G$ has at least $2$ connected components, 
say  $G_1$ and $G_2$. Let $v_i \in V(G_i)$, $i=1,2$. 
Since each of $G-v_1$ and $G-v_2$  either is order-zero graph or has a unique $\gamma$-set, 
$G$ has exactly one $\gamma$-set, which is a contradiction.
Thus $G$ is connected.

To proceed we need the following claim.

\noindent
{\bf Claim 1.} \ 
 If $G = H\circ K_1$,  where $H$  is a  connected graph of order at least $2$, 
 then $V^-(G) = V(G)-V(H)$ and $G$ is not a hypo-$\mathcal{UD}$ graph.  
\begin{proof}[Proof of Claim 1]
Recall that $\gamma (G) = |V(G)|/2$ for any corona $G$ (Theorem \ref{0.5}). 
If $x \in V(H)$ and $y$ is the leaf neighbor of $x$ then (a) $V(H-x)$ is a $\gamma$-set of $G-y$, 
which implies $V(G)-V(H) \subseteq V^-(G)$, and 
(b) $G-x$ is disjoint union of $K_1$ and $ (H-x)\circ K_1$ which leads to $\gamma (G-x) = \gamma (G)$. 
Thus,  $V^-(G) = V(G)-V(H)$. 
Since $V(G)-V(H)$ and $\{y\} \cup V(H-x)$ are $\gamma$-sets of $G-x$, $G$ is not a hypo-$\mathcal{UD}$ graph.
\end{proof}

{\it Case} 1:  $V^-(G) \not= \emptyset$.  
For any  $x \in V^-(G)$ let $D_x$ be the unique $\gamma$-set of $G-x$. 
Then $D_x\cup \{y\}$ is a $\gamma$-set of $G$ for every $y \in N[x]$. 
This implies that $\gamma (G-z) \leq \gamma (G)$ for any $z \in V(G) - D_x$, in particular when $z \in N[x]$. 
Now since $G$ is a hypo-$\mathcal{UD}$ graph, 
(a) $V(G) - (D_x \cup N(x)) \subseteq V^-(G)$, and 
(b) if $x \in V^-(G)$ and $deg(x) \geq 2$ then $N[x] \subseteq V^-(G)$. 

From (b) we conclude that, 
if $G$ has a $\gamma$-critical vertex of degree at least $2$ then $V^-(G) = V(G)$, as required. 
So, let each $\gamma$-critical vertex of $G$ be a leaf. 
Let $x \in V^-(G)$ and $N(x) = \{y\}$.  Since $x$ is a leaf,  $y \not \in V^-(G)$. 
Since $D_x$ is the unique $\gamma$-set of $G-x$, there is no leaf in $D_x$. 
Now by (a),  $D_x \cup \{y\}$ and $V^-(G)$  form a partition of $V(G)$. 
As $D_x \cup \{y\}$ is a $\gamma$-set of $G$, $V^-(G)$ is a dominating set of $G$. 
This implies that each element of $D_x \cup \{y\}$ is adjacent to a leaf. 
Assume that  there is a vertex $z \in D_x$ which is adjacent to at least $2$ leaves.  
	Then $z$ is in all $\gamma$-sets of $G$ which implies that 
	all leaf neighbors of $z$ are outside $V^-(G)$, a contradiction. 
	Thus, $G$ is a corona of a connected graph of order at least $2$. 
	But this is again a contradiction because of Claim 1.

{\it Case} 2: $V^-(G) = \emptyset$. Since $G$ is a hypo-$\mathcal{UD}$ graph, 
there are at least $2$ different $\gamma$-sets of $G$, say $D_1$ and $D_2$. 
If there is $x \in V(G) - (D_1 \cup D_2)$ then since $\gamma(G-x) = \gamma(G)$, 
both $D_1$ and $D_2$ are $\gamma$-sets of $G-x$ - a contradiction.  
Hence $D_1 \cup D_2 = V(G)$ which implies $2\gamma(G) \geq |V(G)|$. 
 By Theorem \ref{0.5},  $2\gamma(G) = |V(G)|$ and either $G$ is a connected corona or $G=C_4$. 
Now Claim 1 and $V(C_4) = V^-(C_4)$ together lead to  $G = K_2$. 
Clearly $K_2$ is a hypo-$\mathcal{UD}$ graph.
\end{proof}
Not all vc-critical graphs are hypo-$\mathcal{UD}$ graphs. 
For example any coalescence $C_{3k+1} \cdot C_{3l+1}$ is
 a vc-critical graph which is not a hypo-$\mathcal{UD}$ graph.
\begin{corollary}\label{minedge}
If $G$ is a hypo-$\mathcal{UD}$ graph of order $n \geq 3$ then $G$ is $2$-edge connected and 
$\delta (G) \geq 2$. Moreover,  
all 	hypo-$\mathcal{UD}$ unicyclic graphs  are $C_{3k+1}$, $k \geq 1$. 
\end{corollary}
\begin{proof}
By Theorem \ref{udvc}, $G$ is a vc-graph. 
Now by Remark \ref{2edge}, $G$ is $2$-edge connected and $\delta (G) \geq 2$. 
Hence if $G$ is unicyclic then $G = C_n$. 
Since all paths $P_m$, $m \geq 2$, having a unique minimum dominating set are 
$P_{3k}$, $k \geq 1$,  it follows that   
$G = C_n$ is a hypo-$\mathcal{UD}$ graph  if and only if  $n = 3k+1$.
\end{proof}
\begin{corollary}\label{min2v}
If $G$ is a hypo-$\mathcal{UD}$ graph of order at least $3$ then 
\begin{itemize}
\item[(i)]  For any $x \in V(G)$, the graph $G-x$ has no $\gamma$-critical vertices.
\item[(ii)] For any pair $x,y$ of  vertices of $G$, $\gamma(G-\{x,y\})  \geq \gamma(G) -1$. 
                    The equality holds at least when $y$ does not belong to the unique $\gamma$-set of $G-x$. 
\end{itemize} 
\end{corollary}
\begin{proof}
If $x \in V(G)$ then $\gamma(G-x) = \gamma (G) -1$ (by Theorem \ref{udvc}). 
Assume that there is $u \in V^-(G-x)$. Then for any $v \in N_{G-x}[u]$ and any $\gamma$-set $D$ of $G-\{x,u\}$,
 the set  $\{v\} \cup D$ is a $\gamma$-set of $G-x$.  
 Since $G-x$ has exactly one $\gamma$-set
 and $\delta (G-x) \geq 1$ (by Corollary \ref{minedge}), we arrive to a contradiction.  
Thus, (i) holds and  for any pair $x,y$ of  vertices of $G$, $\gamma(G-\{x,y\})  \geq \gamma(G) -1$. 
%It remains to noted that 
Finally,  since  the removal  of a vertex which  belongs to no $\gamma$-set
 of a graph has no effect on the domination number, 
 $\gamma(G-\{x,y\})  = \gamma(G) -1$ whenever $y$ does not belong to the unique $\gamma$-set of $G-x$. 
\end{proof}

\begin{proposition}\label{vcbound}
Let $G$ be a connected vc-graph of order $n \geq 4$. 
 Then $\gamma (G) \leq \left\lfloor 2n/5\right\rfloor +1$. 
The equality holds  if and only if $G \in \mathcal{A}$.
\end{proposition}
\begin{proof}
It is easy to check that 
 if $G \in \mathcal{A}$ then $G$ is a vc-graph and $\gamma (G) = \left\lfloor 2n/5\right\rfloor + 1$. 
By Remark \ref{2edge}, if $G$ is a vc-graph then $\delta(G) \geq 2$. 
Now by  Theorem \ref{2/5} we have $\gamma (G) \leq \left\lfloor 2n/5\right\rfloor$ when $G \not \in \mathcal{A}$.  
\end{proof}

\begin{proposition}\label{obud}
Let $G$ be a hypo-$\mathcal{UD}$ graph of order $n$. 
Then $1 \leq \gamma (G) \leq \left\lfloor 2n/5\right\rfloor +1$. Furthermore, 
(i) $\gamma(G) = 1$ if and only if $G=K_2$,  
(ii) $\gamma(G) = 2$ if and only if $n \geq 4$ is even and  $G$ is $K_{n}$   minus a perfect matching,   and 
(iii)   $\gamma (G) = \left\lfloor  2n/5\right\rfloor +1$ if and only if $G \in \{K_2, C_4,C_7\}$.
\end{proposition}
\begin{proof}
By Theorem \ref{udvc}, either $G=K_2$ or $G$ is a connected vc-graph. 
Now $\gamma(K_2) =1$ and Proposition \ref{vcbound} lead to $\gamma (G) \leq \left\lfloor 2n/5\right\rfloor +1$.

(i) Let $G$ be a hypo-$\mathcal{UD}$ graph with $\gamma(G) =1$. 
     Then $G$ has $r \geq 2$ vertices of degree $n-1$. 
		 If $v \in V(G)$ and $deg(v) \leq n-2$ then $G-v$ has $r$ $\gamma$-sets, a contradiction. 
		Thus, $G=K_r$. But clearly, among all complete graphs, only $K_2$  is a hypo-$\mathcal{UD}$-graph. 

(ii) Each vc-graph $G$ with $\gamma(G) = 2$ can be obtained from a complete graph 
      of even order by removing a perfect matching \cite{bcd}. 
			Obviously, every such a graph is a hypo-$\mathcal{UD}$-graph. 
			The result now follows by Theorem \ref{udvc}. 

(iii)  Let $\gamma (G) = \left\lfloor  2n/5\right\rfloor +1$.  
Then either $G=K_2$ or  	$G \in \mathcal{A}$ (by Proposition \ref{vcbound}). 
It is easy to see that among all these graphs only $K_2, C_4$ and $C_7$ are hypo-$\mathcal{UD}$-graphs. 
\end{proof}

\begin{proposition} \label{maxud}
If $G$ is a hypo-$\mathcal{UD}$ $n$-order graph, then
\[
          n \leq (\Delta (G) +1)(\gamma(G) -1) +1.
\]
\end{proposition}
\begin{proof}
By Theorem \ref{udvc}, $G$ is a vc-graph or $G=K_2$. The result now follows by Theorem \ref{vc1}.
\end{proof}
The bound in the above corollary is attainable. This is shown in Proposition \ref{extremall}.

\begin{theorem}\label{bondud}
If $G$ is a hypo-$\mathcal{UD}$ graph then $b(G) \leq \delta(G) + 1$. 
\end{theorem}
\begin{proof}
If $G = K_2$ then the result is obvious. So, let $G$ have at least $3$ vertices. 
 By Theorem \ref{udvc}, $G$ is a vc-graph. 
Denote by $G_x$  the graph obtained from $G$ by removal  of all edges incident to $x \in V(G)$, 
where $deg(x) = \delta(G)$.   
Since $G$ is a hypo-$\mathcal{UD}$ graph, $G_x$ has a unique minimum dominating set. 
Since  $\delta(G) \geq 2$ (by Corollary \ref{minedge}), $G_x$ has edges. 
Lemma \ref{b=1}  now implies that there is an edge of $G_x$, say $e$, such that
 $\gamma(G_x-e) > \gamma(G_x)$. 
 But then $\gamma(G) = \gamma(G-x) +1 = \gamma(G_x) < \gamma(G_x-e)$. 
 Thus $b(G) \leq deg(x) + 1 = \delta(G) +1$. 
\end{proof} 
The bound stated in Theorem \ref{bondud} is tight  at least when $G \in \{C_{3k+1} \mid k \geq 1\}$. 

We now concentrate on  hypo-$\mathcal{ED}$ graphs.

\begin{proposition}\label{obed}
Let $G$ be a hypo-$\mathcal{ED}$ $n$-order graph. 
Then $G$ is connected, $n \geq 4$, and $2 \leq \gamma (G) \leq n/2$.
Furthermore, $\gamma (G) = n/2$ if and only if $G = C_4$. 
\end{proposition}
\begin{proof}
Let $G_1$ and $G_2$ be connected components of $G$ and  $v_i \in V(G_i)$, $i=1,2$. 
Since each of $G-v_1$ and $G-v_2$ has an EDS, $G$ has an EDS - a contradiction.
Thus $G$ is connected. It is easy to check that $C_4$ is the unique hypo-$\mathcal{ED}$ 
graph of oder at most $4$. If $G$ has a vertex of degree $n-1$ then $G$ has an EDS. 
 Hence $\gamma (G) \geq 2$. Finally, by Theorem \ref{0.5} we have that 
 $\gamma (G) \leq n/2$ and if the 
equality holds then either $G$ is $C_4$ or $G$ is a corona of a connected graph.  
Since the set of all leaves of any corona is an EDS, the result immediately follows.
\end{proof}

Next we present a tight  upper bound  on the order of a hypo-$\mathcal{ED}$ graph 
in terms of the domination number and maximum degree of the graph. 
 
\begin{theorem}\label{ minusone}
Let $G$ be a graph without efficient dominating sets. Then $|V(G)| \leq \gamma(G) (\Delta(G) + 1) -1$. 
\begin{itemize}
\item[(i)] Let the equality holds. Then
 (a) for every  $\gamma$-set $D$ of $G$ there is exactly one vertex $y_D \in V(G)-D$ such that $D$ 
        is an efficient dominating set of $G-y_D$ and $y_D$ is adjacent to exactly $2$ vertices in $D$, 
				and 
	(b) each vertex belonging to some $\gamma$-set of $G$ has maximum degree.			
	In particular, if each vertex of $G$ belongs to some $\gamma$-set of $G$ then $G$ is regular. 
\item[(ii)]  If there are a $\gamma$-set $D$ of $G$ and a vertex $y$ of $G-D$  
such that  $D$ is an efficient dominating set of $G-y$, $y$ is adjacent to exactly $2$ vertices of $D$ 
and all vertices of $D$ have maximum degree then  $\gamma(G) (\Delta(G) + 1) -1 = |V(G)|$. 
\end{itemize}	
\end{theorem}

\begin{proof}
Let $D=\{x_1,x_2,..,x_k\}$ be an arbitrary  $\gamma$-set of $G$. 
If $x_ix_j \in E(G)$ then 
\[
|V(G)| \leq \Sigma_{r=1}^k|N[x_r]| - |\{x_i,x_j\}| = \Sigma_{r=1}^k (deg(x_r)+1) - 2 \leq \gamma(G)(\Delta(G)+1) - 2.
\]
If $x_ix_j \not \in E(G)$ and $y \in V(G)-D$ is a common neighbor of both $x_i$ and $x_j$ then 
\[
|V(G)| \leq \Sigma_{r=1}^k|N[x_r]| - |\{y\}| = \Sigma_{r=1}^k (deg(x_r)+1) - 1 \leq \gamma(G)(\Delta(G)+1) - 1.
\]

(i) Suppose $|V(G)| = \gamma(G)(\Delta(G)+1) - 1$.  
Then $|V(G)| = \Sigma_{r=1}^k|N[x_r]| - |\{y\}|$ and $\Sigma_{r=1}^k (deg(x_r)+1) - 1 =\gamma(G)(\Delta(G)+1) - 1$.
Since $D$ is a $\gamma$-set, (a) by the first equality we have that $D$ is independent, each vertex   in $V(G)-(D\cup \{y\})$ 
is adjacent to exactly one vertex of $D$,  and $y$  is adjacent to exactly $2$ vertices in $D$, and 
(b)   by the second equality, it follows that $deg(x_r) = \Delta(G)$ for all $r = 1,2,..,k$. 
The rest is obvious.

(ii) Assume now that there is a $\gamma$-set $D=\{x_1,x_2,..,x_k\}$ of $G$ such  that 
$deg(x_1) = .. =  deg(x_k)  = \Delta(G)$, 
 $D$ is   an efficient dominating set of $G-y$ for some vertex $y \in V(G)-D$ and 
$y$ has exactly $2$ elements of $D$ as neighbors. 
  Then $|V(G)| = \Sigma_{r=1}^k|N_{G-y}[x_r]|+|\{y\}| = \Sigma_{r=1}^k|N_G[x_r]| - |\{y\}| = \gamma(G)(\Delta(G)+1) - 1$.
 \end{proof}

\begin{corollary}\label{min1}
Theorem \ref{ minusone} is valid when $G$ is a hypo-$\mathcal{ED}$ graph.
\end{corollary}

We give the following examples to illustrate the sharpness of the bound in Corollary \ref{min1}. 

\begin{example}\label{cycle}
All hypo-$\mathcal{ED}$ cycles  are $C_{3k+1}$ and $C_{3k+2}$, $k \geq 1$. 
Moreover,   $|V(C_{3k+2})|  =\gamma(C_{3k+2})(\Delta(C_{3k+2}) +1) - 1$, $k \geq 1$.
\end{example}

\begin{example}\label{extr2}
If $G \in  \{C(8k+5, \{1,..,k\} \cup \{3k+2,..,4k+2\}) \mid k \geq 1\}$  
then $G$ is a  hypo-$\mathcal{ED}$ graph with $|V(G)| =\gamma(G) (\Delta(G) + 1) -1$. 
\end{example}
\begin{proof}
First note that $G$ is $(4k+2)$-regular graph of order $8k+5$. 
Hence $\gamma(G) \geq 2$. Since for any $r \in V(G)$ the vertex set 
 $\{r, r+2k+1\}$  is  dominating for $G$ and 
$N[r] \cap N[r+2k+1] = \{r+5k+3\}$, it follows  that $\gamma(G)  = \gamma(G-\{r+5k+3\}) = 2$ 
and $\{r, r+2k+1\}$ is an efficient dominating set for $G-\{r+5k+3\}$  (where addition is taken mod $8k+5$). 
Thus $G$ is a  hypo-$\mathcal{ED}$ graph and clearly $|V(G)| =\gamma(G) (\Delta(G) + 1) -1$ holds.
\end{proof}

\begin{example}\label{extr1}
Let $G \in  \{C(t(2k+1)-1, \{1,..,k\}) \mid k \geq 1, t \geq 2\}$. 
Then $G$ is a  hypo-$\mathcal{ED}$ graph with $|V(G)| =\gamma(G) (\Delta(G) + 1) -1$. 
\end{example}
\begin{proof}
	A graph  $G$ is $2k$-regular  of order $n=t(2k+1)-1$ and by Theorem \ref{circu}, $\gamma(G) = t$. 
Assume first $t$ is odd. Then the set $D_r = \{r\pm l(2k+1)\  (\mbox{mod\ } n) \mid l \in \{0,1,..,(t-1)/2\}\}$
 is a $\gamma$-set of $G$ for any vertex $r$ of $G$. 
Furthermore, the distance between  any pair of distinct vertices of $D_r$ is at least $3$, 
except for the pair $a_1 = r+(t-1)(2k+1)/2$,  $a_2 = r-(t-1)(2k+1)/2$. 
Since $N[a_1]$ and $N[a_2]$ have exactly the vertex $a_1+k$ in common, 
 $D_r -  \{a_1 +k\}$ is an EDS of $G-\{a_1 +k\}$ for any vertex $r$ of $G$. 

 Assume now $t$ is even. Then 
the set $U_r = \{r\pm s(2k+1)\  (\mbox{mod\ } n) \mid s\in \{0,1,..,(t-2)/2\}\} \cup \{t(2k+1)/2 - 1\}$ 
 is a $\gamma$-set of $G$ for any vertex $r$ of $G$. 
Note  that the distance between  any pair of distinct vertices of $U_r$ is at least $3$, 
except for the pair $b_1 = r+(t-2)(2k+1)/2$,  $b_2 = r+t(2k+1)/2 - 1$. 
Since $N[b_1] \cap N[b_2] = \{b_1+k\}$, 
 $U_r -  \{b_1 +k\}$ is an EDS of $G-\{b_1 +k\}$ for any vertex $r$ of $G$. 
\end{proof}

%\textcolor[rgb]{.95,.65,0.05}{\bf $G=C(17, \{2,3,5,7\})$ -- $|V(G)| =\gamma(G) (\Delta(G) + 1) -1$ holds.}

Now we turn our attention to the hypo-$\mathcal{ED}$ graphs having $\gamma$-critical vertices.
%A dominating vertex of a graph is a vertex which is adjacent to every other vertex.

\begin{proposition}\label{ed1}
 A  connected vc-graph $G$ is a hypo-${\mathcal ED}$ graph if and only if 
$G-v$ has an efficient dominating set for all $v \in V(G)$.
\end{proposition}
\begin{proof}
$\Rightarrow$ Obvious.

$\Leftarrow$ If $D$ is an EDS of $G$ and $v \in V(G)-D$ then 
              $D$ is an EDS $G-v$. Now by Theorem \ref{effs1}, 
              $\gamma(G) = |D| = \gamma(G-v)$, a contradiction.  
\end{proof}

\begin{theorem}\label{vced=ud}
Let $G$ be a hypo-$\mathcal{ED}$ vc-graph. 
Then for every vertex $v \in V(G)$, $G-v$ has exactly one efficient dominating set. 
If in addition $G$ is regular then $G$ is a hypo-$\mathcal{UD}$ graph. 
\end{theorem}
\begin{proof}
Let $x \in V(G)$, $D_x$ an EDS of $G-x$, $y \in D_x$ and let $D_y$ be an EDS of $G-y$. 
Note that  $D_y$ and $N[y]$ are disjoint and $|D_y| = \gamma(G-y) = \gamma (G) -1 = \gamma(G-x) = |D_x|$.  
Hence there exists exactly one vertex of $D_y$, say $z$, which is not dominated by $D_x$.                                                      %-\{y\}$. 
But  $D_x$ is a $\gamma$-set of $G-x$. Thus  $z \equiv x$.
As $D_y$ was chosen arbitrarily, $x$ belongs to all EDS of $G-y$. 
By symmetry $y$ belongs to all EDS of $G-x$. 
This allow us to deduce that $D_x$ is the unique EDS of $G-x$.  

Finally, let $G$ be  $k$-regular.  
Then all vertices of $D_x$ have degree $k$ in $G-x$ and $|V(G-x)| = |D_x|(k+1) = \gamma(G-x)(\Delta(G-x)+1)$.
This implies  that all $\gamma$-sets of $G-x$ are efficient dominating.
But we already know that $G-x$ has exactly one EDS. 
Thus $G$ is a hypo-$\mathcal{UD}$ graph.
\end{proof}

\begin{theorem}\label{delta}
Let  a  hypo-${\mathcal ED}$ graph $G$ have a $\gamma$-critical vertex.  
 Then
\[
(\delta(G) + 1)(\gamma(G)-1) +1 \leq |V(G)|  \leq (\Delta(G) +1)(\gamma(G)-1) +1. 
\] 
\end{theorem} 
\begin{proof}
If $x$ is a $\gamma$-critical vertex of $G$ and $D = \{u_1,..,u_k\}$  is an EDS of $G-x$ 
then the sets $\{x\}, N[u_1],..,N[u_k]$ form a partition of $V(G)$.  
 Since $\delta(G)+1 \leq |N[u_i]| = deg (u_i) + 1 \leq \Delta(G) + 1$, $i = 1,..,k$, we have  
\[
1 + (\delta(G) + 1)(\gamma(G)-1) \leq |\{x\}| + \Sigma_{i=1}^k|N[u_i]| = |V(G)| \leq 1 + (\Delta(G) +1)(\gamma(G)-1).
\]
\end{proof}
\begin{corollary}\label{reg}
If $G$  is a  regular hypo-${\mathcal ED}$ graph having a $\gamma$-critical vertex 
 then $|V(G)|  = (\Delta(G) +1)(\gamma(G)-1) +1 = (\delta(G) +1)(\gamma(G)-1) +1$. 
\end{corollary}

\begin{theorem}\label{regiff}
Let $G$ be a connected  graph with $(\Delta (G) + 1) (\gamma(G) -1) + 1 \geq 4$ vertices. 
If $G$ is a vc-graph then $G$ is both a hypo-${\mathcal ED}$ graph and a hypo-${\mathcal UD}$ regular graph. 
\end{theorem}
\begin{proof}
Let $G$ be a vc-graph.  By  Theorem \ref{vc1}, $G$ is regular. 
Let $x \in V(G)$ and $D = \{x_1,..,x_k\}$ a $\gamma$-set of $G-x$.  
Then 
\[
|V(G-x)| \leq \Sigma_{r=1}^k|N[x_r]| = \Sigma_{r=1}^k (\Delta(G)+1) = (\gamma(G)-1)(\Delta(G)+1) = |V(G-x)|. 
\]
Hence $N[x_1], N[x_2],..,N[x_k]$ form a partition of $V(G)$ 
and we can conclude that $D$ is an EDS of $G-x$.  
Thus $G$ is a hypo-${\mathcal ED}$ graph. 
Now by  Theorem \ref{vced=ud} ,  $G$ is a hypo-${\mathcal UD}$ graph. 
\end{proof}

\begin{proposition}\label{extremall}
Let $G  = C(n; \{1,2,..,k\})$, where $n \geq 4$ and  $1\leq k < \left\lfloor n/2 \right\rfloor$.  
Then  $G$ is a hypo-${\mathcal UD}$ graph if and only if $2k+1$ divides $n-1$. 
If $2k+1$ divides $n-1$ then  $n = |V(G)| = (\Delta (G) + 1) (\gamma(G) -1) + 1$, and 
$G$ is  a hypo-${\mathcal ED}$ graph.
\end{proposition}
\begin{proof}
Note that $G$ is a $2k$-regular. First let $2k+1$ divides $n-1$. 
By Theorem \ref{circu}  we have that $n = |V(G)| = (\Delta (G) + 1) (\gamma(G) -1) + 1$ and $G$ is a vc-graph.  
Now $G$ is both a hypo-${\mathcal ED}$ graph and a hypo-${\mathcal UD}$ graph,
 because  Theorem \ref{regiff}. 

If $G$ is a hypo-${\mathcal UD}$ graph then by Theorem \ref{udvc}, $G$ is a vc-graph.  
But then Theorem \ref{circu} implies that $2k+1$ divides $n-1$. 
\end{proof}

%\textcolor[rgb]{.95,.65,0.05}{\bf Remark that for a   hypo-${\mathcal ED}$ graph $G$ obtained from $K_{1,3}$ by subdividing one edge once 
%is fulfilled $(\Delta (G) + 1) (\gamma(G) -1) + 1 = |V(G)|$. However $G$ is neither regular nor vc-critical.}

\section{Open problems and  questions}

We conclude the paper by listing some interesting problems and directions for further research. 
Let $\mathcal{P} \in \{\mathcal{ED}, \mathcal{UD}\}$. 
\begin{itemize}
                    \item[$\bullet$] {\em Find all ordered pairs $(n,k)$ of integers such that there is 
										a hypo-$\mathcal{P}$ graph $G$ of order $n$ and the domination number $k$. }
\end{itemize}

If $\mathcal{P} = \mathcal{UD}$ then by Proposition \ref{obud},  $1 \leq \gamma (G) \leq \left\lfloor 2n/5\right\rfloor +1$. 
Furthermore, 
(a) if $k=1$ then $n=2$, 
(b) if $k=2$ then $n \geq 4$ is even, and 
(c) if  $k= \left\lfloor  2n/5\right\rfloor +1$ then $(n,k) \in \{(2,1), (4,2), (7,3)\}$. 
Note that in \cite{ms}   a  characterization  is  given for  the  connected $n$-order graphs  $G$ 
for which   $\gamma (G) = 2n/5$. 

If $\mathcal{P} =  \mathcal{ED}$ then $2 \leq \gamma(G) \leq n/2$ (Proposition \ref{obed}) and 
moreover if $\gamma(G) = n/2$ then $n=4$.  
A characterization of $n$-vertex connected graphs $G$ 
whose domination number satisfies $\gamma(G) = (n-1)/2$ is obtained in \cite{bchhs}. 
\begin{itemize} 
                   \item[$\bullet$] {\em If $G$ is a hypo-$\mathcal{P}$ graph of order $n$ and the domination number $k$, 
                   what is the maximum/minimum number of edges in  $G$?}
\end{itemize}

\begin{itemize}
                    \item[$\bullet$] {\em Find all hypo-$\mathcal{ED}$ trees and 
                                          all hypo-$\mathcal{ED}$ unicyclic graphs.  }
\end{itemize}

\begin{itemize}
                    \item[$\bullet$] {\em Characterize the hypo-$\mathcal{ED}$ graphs $G$ with $\gamma(G) = 2$.  }
\end{itemize}

\begin{itemize}
                    \item[$\bullet$] {\em Characterize the hypo-$\mathcal{P}$ graphs $G$ for which  
                                          $\overline{G}$ is also a hypo-$\mathcal{P}$ graph. 
                                          In particular, characterize/find all self complementary hypo-$\mathcal{P}$ graphs.} 
\end{itemize}
If both $G$ and  $\overline{G}$ are  hypo-$\mathcal{P}$ graphs then, by Theorem \ref{udvc} and 
Proposition \ref{obed}, it follows that both $G$ and  $\overline{G}$ must be connected. 
Note that  $C_5$ and the {\em bull}  (the graph obtained from $K_3 \circ K_1$ by removing exactly one leaf) 
are self-complementary hypo-$\mathcal{ED}$ graphs.

\begin{itemize}
                    \item[$\bullet$] 
                    {\em Characterize the hypo-$\mathcal{ED}$ graphs $G$ such that $\overline{G}$ has an EDS.}               
\end{itemize}

If a graph $G$ is $K_{2n}$ minus a perfect matching, $n \geq 2$,  then 
we already know that $G$ is a hypo-$\mathcal{ED}$ graph.   
 Since $\overline{G}$ is a union of $n$ copies of $K_2$, $\overline{G}$ has an EDS.

\begin{itemize}
                    \item[$\bullet$] 
                    {\em Characterize the hypo-$\mathcal{UD}$ graphs $G$ such that $\overline{G}$ has a unique $\gamma$-set.}               
\end{itemize}

Brigham et al. \cite{bhhr} defined a graph $G$ to be {\em domination bicritical} 
 if $\gamma(G - S) < \gamma (G)$ for any set $S \subseteq V(G)$ of $2$ vertices.
 \begin{itemize}
                    \item[$\bullet$] \
                    {\em Does there exist a bicritical hypo-$\mathcal{UD}$ graph?}
\end{itemize}

\begin{itemize}
                    \item[$\bullet$] \
                    {\em Does there exist a hypo-$\mathcal{UD}$ graph with a cut-vertex?}
\end{itemize}

	A graph $G$ is $\gamma$-{\em EA-critical} if $\gamma(G+e) < \gamma (G)$ for each edge $e \in E(\overline{G})$. 
	Clearly if $G$   is $K_{2n}$ minus a perfect matching, $n \geq 2$, 
	then $G$ is  a hypo-$\mathcal{P}$  $\gamma$- EA-critical graph. 
	
	\begin{itemize}
                    \item[$\bullet$] \
                    {\em Find results on  the hypo-$\mathcal{P}$  $\gamma$-EA-critical graphs.}               
\end{itemize}
	
		\begin{itemize}
                    \item[$\bullet$] 
                    {\em Is it true that  $b(G) = \delta (G) + 1$ for each hypo-$\mathcal{UD}$ graph?}               
\end{itemize}
	
		Let $G$ be a graph and let  $\mathcal{L}$ and  $\mathcal{R}$ be arbitrary graph-properties. 
We define a dominating set $D$ of a graph $G$ to be  a  {\em dominating $(\mathcal{L},\mathcal{R})$-set} of $G$ 
if $D$ is a $\mathcal{L}$-set of $G$ and $V(G)-D$ is a $\mathcal{R}$-set of $G$. 
We define the  {\em domination number with respect to the ordered pair} $(\mathcal{L},\mathcal{R})$ {\em of graph-properties}, 
		denoted by  $\gamma_{(\mathcal{L},\mathcal{R})} (G)$ 		
to be  the smallest cardinality of a dominating $(\mathcal{L},\mathcal{R})$-set of $G$. 
	   A dominating $(\mathcal{L},\mathcal{R})$-set of $G$ with cardinality  $\gamma_{(\mathcal{L},\mathcal{R})}(G)$ is called a 
   $\gamma_{(\mathcal{L},\mathcal{R})}$-{\em set} of $G$. 
	Clearly $\gamma_{(\mathcal{I},\mathcal{I})} \equiv \gamma$. 
	Among the many examples of such numbers one can find in the literature are 
	the independent/total/connected/acyclic/paired/restrained/total-restrained/outer-connected domination numbers. 
	For details, see e.g. \cite{hhs1, hhs2}. 
	 We define a graph $G$ to be a  {\em hypo-unique  $(\mathcal{L},\mathcal{R})$-domination graph} if 
   $G$ has at least two  $\gamma_{(\mathcal{L},\mathcal{R})}(G)$-sets, but $G-v$ has 
	a unique minimum dominating $(\mathcal{L},\mathcal{R})$-set  for each $v\in V(G)$.  

\begin{itemize}
                    \item[$\bullet$] 
                    {\em Find results on  the  hypo-unique  $(\mathcal{L},\mathcal{R})$-domination graphs.}               
\end{itemize}

\end{document}